\documentclass[11pt]{amsart}

\usepackage{amssymb}

\newtheorem{thm}{Theorem}[section]
\newtheorem{lem}[thm]{Lemma}
\newtheorem{cor}[thm]{Corollary}
\newtheorem{prop}[thm]{Proposition}

\theoremstyle{definition}

\theoremstyle{remark}
\newtheorem{rem}[thm]{Remark}

\numberwithin{equation}{section}

\makeatletter
\@namedef{subjclassname@2020}{\textup{2020} Mathematics Subject Classification}
\makeatother

\newcommand{\set}[1]{\left\{#1\right\}}
\newcommand{\norm}[1]{\|#1\|}
\newcommand{\abs}[1]{|#1|}

\newcommand{\N}{\mathbb N}
\newcommand{\R}{\mathbb R}
\newcommand{\C}{\mathbb C}

\newcommand{\D}{\mathbb D}
\newcommand{\DD}{\overline{\mathbb D}}

\newcommand{\hol}{\mathrm{Hol}}
\newcommand{\Ker}{\mathrm{Ker\,}}
\newcommand{\Ran}{\mathrm{Ran\,}}
\newcommand{\aut}{\mathrm{Aut\,}}

\newcommand{\seq}[1]{(#1(k))_{k=0}^{\infty}}
\newcommand{\har}[1]{\mathrm{H}^2_{#1}}
\newcommand{\mult}{\mathrm{Mult}}

\begin{document}

\title{On the similarity of powers of operators with flag structure}


\author{Jianming Yang}
\address{Department of MathematicsHebei Normal University 050016 Shijiazhuang,
P.R.China}
\curraddr{curraddr}
\email{nova\_yang@petalmail.com}

\author{Kui Ji}
\address{Department of MathematicsHebei Normal University 050016 Shijiazhuang,
P.R.China}
\curraddr{curraddr}
\email{jikui@hebtu.edu.cn}


\date{Dec.27th, 2023}



\begin{abstract}
  Let $\mathrm{L}^2_a(\D)$ be the classical Bergman space and denote $M_h$ for the multiplication operator by a function $h$.
  Let $B$ be a finite Blaschke product with order $n$.
  An open question proposed by R. G. Douglas is whether the operators $M_B$ on $\mathrm{L}^2_a(\D)$ similar to $\oplus_1^n M_z$ on $\oplus_1^n \mathrm{L}^2_a(\D)$?
  The question was answered in the affirmative, not only for Bergman space but also for many other Hilbert spaces with  reproducing kernels.
  Since the operator $M_z^*$  is in Cowen-Douglas class $B_1(\D)$ in many cases, Douglas's question can be expressed as a version for operators in $B_1(\D)$, and it is affirmative for many operators in $B_1(\D)$.
  A natural question is how about Douglas's question in the version for operators in Cowen-Douglas class $B_n(\D)$ ($n>1$)?
  In this paper, we investigate a family of operators, which are in a norm dense subclass of Cowen-Douglas class $B_2(\D)$, and give a negative answer.
  This indicates that Douglas's question cannot be directly generalized to general Hilbert spaces with vector-valued analytical reproducing kernel.
\end{abstract}

\maketitle

\section{Introduction}

  Let $\D$ be the open unit disk in $\C$ and $\DD$ be the closed unit disk. A function $B$ on $\D$ is said a finite Blaschke product with order $n$ if it has the following form
  $$ B(z) = e^{i \theta} \prod_{j=1}^{n} \frac{z-a_j}{1-\bar{a_j}z},\ z \in \D, $$
  for some $\theta \in [0,2\pi)$ and some $a_1,\dots,a_n \in \D$. Let $\mathrm{L}^2_a(\D)$ be the classical Bergman space on $\D$ and denote $M_h$ for the multiplication operator on $\mathrm{L}^2_a(\D)$ by a function $h \in \hol(\D)$. A previously open question proposed by R. G. Douglas (Question 6 in \cite{Douglas2007}) is whether the operator $M_B$ on $\mathrm{L}^2_a(\D)$ similar to $\oplus_1^n M_z$ on $\oplus_1^n \mathrm{L}^2_a(\D)$, where $n$ is the order of the finite Blaschke product $B$?

  The question was answered in the affirmative (see \cite{Jiang2007} or \cite{GUO2011}).
  Later, the question was also answered in the affirmative on many other analytic function Hilbert spaces, such as weighted Bergman spaces $\mathrm{A}^2_{\alpha}$ (see \cite{JIANG2010}), Sobolev disk algebra $\mathrm{R}(\D)$ (see \cite{Wang2009} or \cite{Ji2013}), and Dirichlet space $\mathfrak{D}$ (see \cite{Li2022}). Recently in \cite{hou2023}, Hou and Jiang proved that the question still holds on the weighted Hardy space of polynomial growth, which covers the weighted Bergman space, the weighted Dirichlet space, and many weighted Hardy spaces defined without measures.

  Douglas's question originated in the study of reducing subspaces of analytic multiplication operators on Bergman space. Before the question was proposed, there had been many studies of the reducing subspaces of the multiplication operator by a finite Blaschke product (see \cite{Zhu2000}, \cite{Hu2004}, and \cite{Guo2009}). An application of Douglas's question, together with the technology of strongly irreducible operator and $\mathrm{K}_0$-group, is the similarity of analytic Toeplitz operators (e.g. Theorem 1.1 in \cite{JIANG2010}), which is analogous to the result on Hardy space (see \cite{COWEN1982}).

  Note that $M_B$ is also equal to the analytic function calculus $B(M_z)$, then we can describe Douglas's question in the following general form: let $T$ be a certain bounded linear operator on a complex separable Hilbert space $H$ and suppose $\sigma(T) \subseteq \DD$, then for any finite Blaschke product $B$, is the operator $B(T)$ similar to $\oplus_1^n T$ on $\oplus_1^n H$, where $n$ is the order of $B$?

Obviously,  an operator $T$ satisfies Douglas's question if and only if so does its adjoint $T^*$. Since the adjoints $M_z^*$ of the multiplication operators $M_z$ by $z$ on many analytic function Hilbert spaces are in Cowen-Douglas class $B_1(\D)$ proposed in \cite{Cowen1978}, the works mentioned earlier correspond to Douglas's question in the case that $T$ is in $B_1(\D)$. In that way, how about Douglas's question in the case that $T$ is in Cowen-Douglas class $B_n(\D)(n>1)$?

  For a long time, there has been a lack of sufficient understanding of Cowen-Douglas class $B_n(\Omega)$ for the higher rank case. A new but important subclass $FB_n(\Omega)$ has been introduced by G. Misra  in \cite{JI2017}.
  All irreducible homogeneous operators in $B_n(\D)$ are in $FB_n(\D)$ (see \cite{JI2017}), and the class $FB_n(\Omega)$ (or even its subclass $CFB_n(\Omega)$) is norm dense in $B_n(\Omega)$ (see \cite{Jiang2023}).  G. Misra etc. proved the operator in $FB_n(\Omega)$  possesses a flag structure. It is also proved that the flag structure is rigid, that is, the unitary equivalence class of the operator and the flag structure determine each other (see \cite{JI2017}).

  In this paper, we investigate a family of operators in class $FB_2(\D)$ and give a negative answer to Douglas's question.

  Denote by $\hol(\D)$ the set of all analytical functions on $\D$ and denote by $\hol(\DD)$ the set of all analytical functions on $\DD$. Let $\har{\alpha}$ be a weighted Hardy space with weight sequence $\alpha$ on $\D$ and denote $M_{\alpha,h}$ for the multiplication operator by $h \in \hol(\D)$ on $\har{\alpha}$. Let $\har{\beta}$ be another weighted Hardy space with weight sequence $\beta$ on $\D$ and denote $M_{\alpha,\beta,h}$ for the multiplication operator by $h \in \hol(\D)$ from $\har{\alpha}$ to $\har{\beta}$. These concepts will be introduced in detail later.

  The following theorem is the main theorem of this paper.

\begin{thm}\label{thm 4.1}
  Let $\har{\alpha}$ and $\har{\beta}$ be two M$\ddot{o}$bius invariant weighted Hardy spaces with property A for some integer $n_0 \geq 2$. Suppose the weight sequences $\alpha$ and $\beta$ satisfy
  $$ \sup_{k \geq 0} \frac{\beta(k)}{\alpha(k)} < \infty \quad \text{and} \quad \lim_{k \to \infty} \frac{\alpha(k)}{k\beta(k)} = 0. $$
  Denote $T$ be the operator
  $$ \begin{pmatrix}
       M_{\alpha,z}^* & M_{\alpha,\beta,h}^* \\
       0 & M_{\beta,z}^* \\ \end{pmatrix} $$
  on $\har{\alpha} \oplus \har{\beta}$, where $h \in \hol(\DD)$.

  If for any finite Blaschke product $B$, $B(T)$ is similar to $\oplus_1^n T$, where $n$ is the order of $B$, then $h=0$.
\end{thm}

  In other words, under the hypothesis of Theorem \ref{thm 4.1}, for any non-zero function $h$ in $\hol(\DD)$, the operator
  $$ \begin{pmatrix}
    M_{\alpha,z}^* & M_{\alpha,\beta,h}^* \\
    0 & M_{\beta,z}^* \end{pmatrix} $$
  doesn't satisfy Douglas's question.

\section{Preliminaries}

\subsection{Weighted Hardy space}

  In the present article, we denote the power series coefficients of an analytical function $f$ as $\seq{\hat{f}}$, i.e. $f(z)=\sum_{k=0}^{\infty} \hat{f}(k) z^k$.

\subsubsection{}

  Let $\alpha = \seq{\alpha}$ be a positive sequence, then the weighted Hardy space $\har{\alpha}$ with weight sequence $\alpha$ on $\D$ is defined as
  $$ \set{f \in \hol(\D) \colon \sum_{k=0}^{\infty} \abs{\hat{f}(k)}^2 \alpha(k)^2 < \infty}. $$ $\har{\alpha}$ has an inner product $\langle \cdot , \cdot \rangle_\alpha$ defined as
  $$ \langle f,g \rangle_\alpha = \sum_{k=0}^{\infty} \hat{f}(k) \overline{\hat{g}(k)} \alpha(k)^2, $$
  where $f,g \in \har{\alpha}$. \textbf{In the present article, we always assume the weight sequence $\alpha$ satisfying}
  \begin{equation}\label{equation 1.1}
    \lim_{k \to \infty} \frac{\alpha(k+1)}{\alpha(k)} = 1.
  \end{equation}
  In this case, $\har{\alpha}$ is a Hilbert space and contains $\hol(\DD)$. What's more, if a complex sequence $a = \seq{a}$ satisfies $\sum_{k=0}^{\infty} \abs{a(k)}^2 \alpha(k)^2 < \infty$ and a function $f$ is defined as $f(z) = \sum_{k=0}^{\infty} a(k) z^k$, then $f \in \hol(\D)$ and hence $f \in \har{\alpha}$.

  The norm of $f \in \har{\alpha}$ is $\norm{f}_\alpha = (\sum_{k=0}^{\infty} \abs{\hat{f}(k)}^2 \alpha(k)^2)^{\frac{1}{2}}$. Let $e_k(z)=z^k$, $z \in \D$, then $\{ e_k \}_{k=0}^{\infty}$ forms an orthogonal basis of $\har{\alpha}$ (thus, $\har{\alpha}$ is separable). In addition, $\norm{e_k}_\alpha = \alpha(k)$ and $f = \sum_{k=0}^{\infty} \hat{f}(k) e_k$ in the sense of norm $\norm{\cdot}_\alpha$. 

  It can be checked that for each $\omega \in \D$, the linear function $\delta_\omega \colon f \in \har{\alpha} \to f(\omega) \in \C$ is continuous. Then by Riesz representation Theorem, $\har{\alpha}$ has a reproducing kernel $\{ k_\omega \}_{\omega \in \D}$, i.e. for each $\omega \in \D$, $k_\omega \in \har{\alpha}$ and $\langle f,k_\omega \rangle_\alpha = f(\omega)$ whenever $f \in \har{\alpha}$.

\subsubsection{}

  Let $\har{\alpha}$ be a weighted Hardy space. The multiplier algebra $\mult(\har{\alpha})$ of $\har{\alpha}$ is defined as the set of all $h \in \hol(\D)$ such that $h f \in \har{\alpha}$ for any $f \in \har{\alpha}$.
  For each $h \in \mult(\har{\alpha})$, one can define a multiplication operator $M_{\alpha,h} \colon f \in \har{\alpha} \to h f \in \har{\alpha}$, which is bounded by closed graph Theorem. Sometimes $M_{\alpha,h}$ is also abbreviated as $M_{h}$. As well known, $\mult(\har{\alpha}) \subseteq \mathrm{H}^{\infty}(\D) \cap \har{\alpha}$, where $ \mathrm{H}^{\infty}(\D)$ is the set of all bounded analytic functions on $\D$. The Mapping $h \in \mult(\har{\alpha}) \to M_{\alpha,h} \in \mathcal{L}(\har{\alpha})$ is an algebraic monomorphism with $1 \to I$, and for each $h \in \mult(\har{\alpha})$, $M_{\alpha,h}$ is invertible if and only if $\frac{1}{h} \in \mult(\har{\alpha})$.

  As well known, the assumption (\ref{equation 1.1}) guarantees that the multiplication operator $M_{\alpha,z}$ is a bounded linear operator on $\har{\alpha}$ and $\norm{M_{\alpha,z}} = \sup_{k \geq 0} \frac{\alpha(k+1)}{\alpha(k)}$. Denote $\{ M_{\alpha,z} \}'$ for the commutant algebra of $M_{\alpha,z}$, i.e. the set of all $X \in \mathcal{L}(\har{\alpha})$ such that $X M_{\alpha,z} = M_{\alpha,z} X$.
  Then it can be verified that $X$ is in $\{ M_{\alpha,z} \}'$ if and only if $X$ is equal to $M_{\alpha,h}$ for some $h \in \mult(\har{\alpha})$.

  Furthermore, the assumption (\ref{equation 1.1}) also guarantees that $\hol(\DD) \subseteq \mult(\har{\alpha})$ (see Lemma 3.4 in \cite{Fricain2013} or see \cite{hou2023}). Then it can be verified that the spectrum $\sigma(M_{\alpha,z})$ of $M_{\alpha,z}$ is contained in $\DD$ and for any $h \in \hol(\DD)$, the analytic function calculus $h(M_{\alpha,z})$ is $M_{\alpha,h}$.


  Let $\har{\alpha}$ and $\har{\beta}$ be two weighted Hardy spaces. The multiplier $\mult(\har{\alpha},\har{\beta})$ from $\har{\alpha}$ to $\har{\beta}$ is defined as the set of all $h \in \hol(\D)$ such that $h f \in \har{\beta}$ for any $f \in \har{\alpha}$.
  For each $h \in \mult(\har{\alpha},\har{\beta})$, one can define a multiplication operator $M_{\alpha,\beta,h} \colon f \in \har{\alpha} \to h f \in \har{\beta}$, which is bounded by closed graph Theorem. Sometimes $M_{\alpha,\beta,h}$ is also abbreviated as $M_{h}$. Obviously, $\mult(\har{\alpha},\har{\beta}) \subseteq \har{\beta}$.

  In addition, it can also be verified that $X \in \mathcal{L}(\har{\alpha},\har{\beta})$ satisfies $X M_{\alpha,z} = M_{\beta,z} X$ if and only if $X$ is equal to $M_{\alpha,\beta,h}$ for some $h \in \mult(\har{\alpha},\har{\beta})$.

\begin{rem}
  A simple example of weighted Hardy space involved in the present article is the Hilbert space $\mathrm{H}_{(\lambda)}$. For any $\lambda \in \R$, $\mathrm{H}_{(\lambda)}$ is defined as the weighted Hardy space $\har{\alpha_{\lambda}}$ with weight sequence $\alpha_{\lambda}(k) = (k+1)^{\lambda}$, $k=0,1,\dots$.
  This type of space contains many classical analytic function spaces on $\D$. For example, the classical Hardy space $\mathrm{H}^2(\D)$ ($\lambda = 0$), the classical Bergman space $\mathrm{L}^2_a(\D)$ ($\lambda = -\frac{1}{2}$) and the classical Dirichlet space $\mathfrak{D}$ ($\lambda = \frac{1}{2}$).
\end{rem}

\begin{rem}
  A complicated example of weighted Hardy space involved in the present article is the weighted Hardy space of polynomial growth (which has been studied in \cite{hou2023} recently). A weighted Hardy space $\har{\alpha}$ (usually assume $\alpha(0)=1$) is called of polynomial growth if the weight sequence $\alpha$ satisfies
  $$ \sup_{k \in \N} (k+1) \abs{\frac{\alpha(k)}{\alpha(k-1)} - 1} < \infty. $$
  The weighted Hardy spaces of polynomial growth cover the weighted Bergman space, the weighted Dirichlet space, and many weighted Hardy spaces defined without measures. 
\end{rem}

\subsection{Cowen-Douglas operators}

  Recall some basic concepts of the Cowen-Doug-las operator class $B_n(\Omega)$, which was introduced in \cite{Cowen1978}, and the subclass $FB_2(\Omega)$ of $B_2(\Omega)$, which was introduced in \cite{JI2017}.

  Let $H$ be a complex separable Hilbert space and $\mathcal{L}(H)$ denote the collection of bounded linear operators on $H$. For $\Omega$ a connected open subset of $\C$ and $n$ a positive integer, let $B_n(\Omega)$ denote the operators $T$ in $\mathcal{L}(H)$ which satisfy:

  (a) $\Omega \subseteq \sigma(T) = \set{\omega \in \C \colon T-\omega \ \text{not invertible}}$;

  (b) $\Ran(T-\omega) = H$ for $\omega \in \Omega$;

  (c) $\vee_{\omega \in \Omega}\, \Ker(T-\omega) = H$; and

  (d) $\dim \Ker(T-\omega) =n$ for $\omega \in \Omega$.

  As well known, the adjoint $M_{\alpha,z}^*$ of the multiplication operator $M_{\alpha,z}$ acting on the weighted Hardy space $\har{\alpha}$ is in class $B_1(\D)$.

  The operator class $FB_2(\Omega)$ is the set of all bounded linear operators $T$ of the form $$\begin{pmatrix}
  T_0 & S \\
  0 & T_1 \\
  \end{pmatrix},$$
  where $T_0$ and $T_1$ are in class $B_1(\Omega)$ and the operator $S$ is a non-zero intertwiner between them, i.e. $T_0 S = S T_1$.

  Let $\har{\alpha}$ and $\har{\beta}$ be two weighted Hardy spaces. Throughout this paper, we will denote $T_{\alpha,\beta,h}$ as the operator
  $$ \begin{pmatrix}
       M_{\alpha,z}^* & M_{\alpha,\beta,h}^* \\
       0 & M_{\beta,z}^* \\ \end{pmatrix} $$
  on $\har{\alpha} \oplus \har{\beta}$ for a function $h$ in $\mult(\har{\alpha},\har{\beta})$.
One can see  that $T_{\alpha,\beta,h}$ is in class $FB_2(\D)$ whenever $h$ is a non-zero function.

\section{Two propositions}\label{step 1}

  Let $\har{\alpha}$ and $\har{\beta}$ be two weighted Hardy spaces. Let $T = T_{\alpha,\beta,h}$, where $h$ is a non-zero function in $\mult(\har{\alpha},\har{\beta})$. To investigate whether for any finite Blaschke product $B$, $B(T)$ is similar to $\oplus_1^n T$, where $n$ is the order of $B$, we will treat $B$ in two cases: the case of $B(z) = e^{i \theta} \frac{z-a}{1-\bar{a}z}$ and the case of $B(z) = z^n$. In \ref{step 2}, we deal with the first case and the main point is Proposition \ref{prop 3.5}. In \ref{step 3}, we deal with the second case and the main points are Proposition \ref{prop 3.6} and Proposition \ref{prop 3.9}. Before that, in \ref{step 1}, we will first prove the following two propositions: Proposition \ref{prop 3.2} and Proposition \ref{prop 3.4}.

  In this paper, we write $T \sim \tilde{T}$ for two bounded linear operators $T$ and $\tilde{T}$ if $T$ is similar to $\tilde{T}$ (i.e. there exists an invertible operator $X$ such that $T X = X \tilde{T}$).


  Let $T_0$ and $T_1$ be two bounded linear operators on Hilbert spaces $H_0$ and $H_1$ respectively. Denote $\sigma_{T_0,T_1} (X) := T_0 X-X T_1$ for any $X \in \mathcal{L}(H_1,H_0)$. Then a linear operator $\sigma_{T_0,T_1} \colon \mathcal{L}(H_1,H_0) \to \mathcal{L}(H_1,H_0)$ can be defined. Let $\sigma_{T}$ be the operator $\sigma_{T,T}$.

  From Lemma 2.18 in \cite{JI2017} and Theorem 2.19 in \cite{JI2017}, a useful conclusion can be obtained that for any $T$ in class $B_1(\Omega)$, $\Ker\sigma_{T} \cap \Ran\sigma_{T} = \set{0}$. But we need to slightly generalize this.

\begin{prop}\label{prop 3.2}
  Let $\har{\alpha}$ and $\har{\beta}$ be two weighted Hardy spaces. If the weight sequences $\alpha$ and $\beta$ satisfy condition
  \begin{equation}\label{cond 1}
    \lim_{k \to \infty} \frac{\alpha(k)}{k\beta(k)} = 0,
  \end{equation}
  then $\Ker\sigma_{M_{\alpha,z}^*,M_{\beta,z}^*} \cap \Ran\sigma_{M_{\alpha,z}^*,M_{\beta,z}^*} = \set{0}$.
\end{prop}

\begin{proof}
  Clearly, we just need to explain that for any $Y \in \mathcal{L}(\har{\alpha},\har{\beta})$, the equations $Y M_{\alpha,z} = M_{\beta,z} Y$ and $Y = X M_{\alpha,z} - M_{\beta,z} X$, for some $X \in \mathcal{L}(\har{\alpha},\har{\beta})$, imply $Y=0$.

  The equation $Y M_{\alpha,z} = M_{\beta,z} Y$ is equivalent to $Y = M_{\alpha,\beta,h}$ for some $h \in \mult(\har{\alpha},\har{\beta})$. Then we have $Y = X M_{\alpha,z} - M_{\beta,z} X$, $X M_{\alpha,z} = M_{\beta,z} X + M_{\alpha,\beta,h}$. Hence, $X z^{k} = z X z^{k-1} + h z^{k}$, $k=1,2,\dots$. Let $g := X z^{0} \in \har{\beta}$. Then by induction, we have that $X z^{k} = z^{k} g + k z^{k-1} h$, $k=1,2,\dots$.

  Let $X z^{k} = \sum_{j=0}^{\infty} x_{jk} z^{j}$, $k=0,1,\dots$. Then
  \begin{equation}\label{3.2.1}
    \langle X z^{k},z^{l}\rangle
    = \langle \sum_{j=0}^{\infty} x_{jk} z^{j},z^{l}\rangle
    = \sum_{j=0}^{\infty} x_{jk} \langle z^{j},z^{l}\rangle
    = x_{lk} \beta(l)^2,
  \end{equation}
  where $k,l=0,1,\dots$. On the other hand,
  $$z^{k} g(z) = \sum_{j=0}^{\infty} \hat{g}(j) z^{j+k} = \sum_{j=0}^{\infty} \hat{g}(j-k) z^{j},$$
  $$z^{k} h(z) = \sum_{j=0}^{\infty} \hat{h}(j) z^{j+k} = \sum_{j=0}^{\infty} \hat{h}(j-k) z^{j},$$
  where $k = 0,1,\dots$, and $\hat{g}(j) = 0 = \hat{h}(j)$ whenever $j<0$. Hence,

  \begin{equation}\label{3.2.2}
    \begin{split}
      \langle X z^{k},z^{l}\rangle
      &= \langle z^{k} g + k z^{k-1} h,z^{l}\rangle\\
      &= \langle z^{k} g,z^{l}\rangle + k \langle z^{k-1} h,z^{l}\rangle\\
      &= \langle \sum_{j=0}^{\infty} \hat{g}(j-k) z^{j},z^{l}\rangle + k \langle \sum_{j=0}^{\infty} \hat{h}(j-k+1) z^{j},z^{l}\rangle\\
      &= \hat{g}(l-k) \beta(l)^2 + k \hat{h}(l-k+1) \beta(l)^2,
  \end{split}
  \end{equation}
  where $k=1,2,\dots$ and $l=0,1,\dots$. Note that $X z^0 = g$, equation (\ref{3.2.2}) still holds for $k=0$. From (\ref{3.2.1}) and (\ref{3.2.2}), it follows that
  $$ x_{lk} = \hat{g}(l-k) + k \hat{h}(l-k+1), $$
  where $k,l=0,1,\dots$.

  Let $t=0,1,\dots$, $k=1,2,\dots$, and $l=k+t-1$. Then $ x_{k+t-1,k} = \hat{g}(t-1) + k \hat{h}(t) $.  Hence,
  \begin{equation}\label{3.2.3}
    \frac{x_{k+t-1,k}}{k} = \frac{\hat{g}(t-1)}{k} + \hat{h}(t) \to \hat{h}(t),\ k \to \infty,\ \text{for}\ t=0,1,\dots.
  \end{equation}
  From condition (\ref{cond 1}) and assumption (\ref{equation 1.1}), we can see that
  $$ \frac{\alpha(k)}{k\beta(k+t-1)} \to 0,\ k \to \infty,\ \text{for}\ t=0,1,\dots. $$
  From (\ref{3.2.1}), it follows that
  \begin{equation*}
    \begin{split}
      \abs{x_{k+t-1,k}}
      &= \frac{ \abs{\langle X z^{k},z^{k+t-1}\rangle} }{\beta(k+t-1)^2} \\
      &\leq \frac{ \norm{X} \norm{z^{k}}_\alpha \norm{z^{k+t-1}}_\beta }{\beta(k+t-1)^2} \\
      &= \frac{\alpha(k)}{\beta(k+t-1)} \norm{X}.
    \end{split}
  \end{equation*}
  Hence,
  \begin{equation}\label{3.2.4}
    \frac{\abs{x_{k+t-1,k}}}{k} \leq \frac{\alpha(k)}{k\beta(k+t-1)} \norm{X} \to 0,\ k \to \infty,\ \text{for}\ t=0,1,\dots.
  \end{equation}
  From (\ref{3.2.3}) and (\ref{3.2.4}), it can be observed that $\hat{h}(t)=0$ when $t=0,1,\dots$. This implies $h=0$ and consequently $Y=M_{\alpha,\beta,h}=0$.
\end{proof}

\begin{rem}\label{example 1 rem 1}
  Let $\lambda,\mu \in \R$. The weight sequences of the Hilbert spaces $\mathrm{H}_{(\lambda)} = \har{\alpha_{\lambda}}$ and $\mathrm{H}_{(\mu)} = \har{\alpha_{\mu}}$ are $\alpha_{\lambda}(k) = (k+1)^{\lambda}$ and $\alpha_{\mu}(k) = (k+1)^{\mu}$, respectively. Then the weight sequences $\alpha = \alpha_{\lambda}$ and $\beta = \alpha_{\mu}$ satisfy the condition (\ref{cond 1}) if and only if $\lambda - \mu < 1$.
\end{rem}

 The following proposition, which can be used to deal with two similar operators in $FB_2(\Omega)$, is Proposition 3.3 in \cite{JI2017}.

\begin{lem}\label{lem 3.1}
  If $X$ is an invertible operator intertwining two operators $T$ and $\tilde{T}$ in $FB_2(\Omega)$, i.e. $XT=\tilde{T}X$, then $X$ is upper triangular:
  $$ X=\begin{pmatrix}
    X_{11} & X_{12} \\
    0 & X_{22} \\ \end{pmatrix}, $$
  and the operators $X_{11}$ and $X_{22}$ are invertible.
\end{lem}

\begin{prop}\label{prop 3.4}
  Let $\har{\alpha}$ and $\har{\beta}$ be two weighted Hardy spaces, where the weight sequences $\alpha$ and $\beta$ satisfy condition (\ref{cond 1}). If the operators $T_{\alpha,\beta,h}$ and $T_{\alpha,\beta,\tilde{h}}$ are similar, where $h$ and $\tilde{h}$ are two non-zero functions in $\mult(\har{\alpha},\har{\beta})$, then there exist $h_1 \in \mult(\har{\alpha})$ and $h_2 \in \mult(\har{\beta})$ such that $\frac{1}{h_1} \in \mult(\har{\alpha})$, $\frac{1}{h_2} \in \mult(\har{\beta})$, and $h = \tilde{h} h_1 h_2$.
\end{prop}

\begin{proof}
  Since $T := T_{\alpha,\beta,h} \sim \tilde{T} := T_{\alpha,\beta,\tilde{h}}$, there exists an invertible operator $X$ such that $XT=\tilde{T}X$. According to Lemma \ref{lem 3.1} and $T$ and $\tilde{T}$ are in $FB_2(\D)$, it follows that $X$ is upper triangular:
  $$ X=\begin{pmatrix}
    X_{11} & X_{12} \\
    0 & X_{22} \\ \end{pmatrix}, $$
  and the operators $X_{11}$ and $X_{22}$ are invertible.

  From $XT=\tilde{T}X$, we have
  $$ X_{11} M_{\alpha,z}^* = M_{\alpha,z}^* X_{11}, $$
  $$ X_{22} M_{\beta,z}^* = M_{\beta,z}^* X_{22}, $$
  $$ X_{11} M_{\alpha,\beta,h}^* + X_{12} M_{\beta,z}^* = M_{\alpha,z}^* X_{12} + M_{\alpha,\beta,\tilde{h}}^* X_{22}. $$
  The first two equations show that $X_{11}=M_{\alpha,g_1}^*$ and $X_{22}=M_{\beta,g_2}^*$ for some $g_1 \in \mult(\har{\alpha})$ and $g_2 \in \mult(\har{\beta})$. The third equation shows that
  $$ X_{11} M_{\alpha,\beta,h}^* - M_{\alpha,\beta,\tilde{h}}^* X_{22} = \sigma_{M_{\alpha,z}^*,M_{\beta,z}^*} (X_{12}). $$

  Since $M_{\alpha,g_1}^*=X_{11}$ and $M_{\beta,g_2}^*=X_{22}$ are invertible, we have $\frac{1}{g_1} \in \mult(\har{\alpha})$ and $\frac{1}{g_2} \in \mult(\har{\beta})$. What's more, we can directly calculate
  $$ \sigma_{M_{\alpha,z}^*,M_{\beta,z}^*} (X_{11} M_{\alpha,\beta,h}^* - M_{\alpha,\beta,\tilde{h}}^* X_{22}) = 0. $$

  Therefore, $X_{11} M_{\alpha,\beta,h}^* - M_{\alpha,\beta,\tilde{h}}^* X_{22}$ is in $\Ran\sigma_{M_{\alpha,z}^*,M_{\beta,z}^*} \cap \Ker\sigma_{M_{\alpha,z}^*,M_{\beta,z}^*}$. According to Proposition \ref{prop 3.2}, this implies $X_{11} M_{\alpha,\beta,h}^* = M_{\alpha,\beta,\tilde{h}}^* X_{22}$ and consequently $M_{\alpha,\beta,h} M_{\alpha,g_1} = M_{\beta,g_2} M_{\alpha,\beta,\tilde{h}}$. Thus, $h g_1 = g_2 \tilde{h}$.

  Finally, take $h_1 = \frac{1}{g_1}$ and $h_2 = g_2$, and then we complete the proof.
\end{proof}

\begin{rem}\label{prop 3.4 rem}
  Let $\har{\alpha}$ and $\har{\beta}$ be two weighted Hardy spaces. Let $h$ and $\tilde{h}$ be in $\mult(\har{\alpha},\har{\beta})$. If there exist $h_1 \in \mult(\har{\alpha})$ and $h_2 \in \mult(\har{\beta})$ such that $\frac{1}{h_1} \in \mult(\har{\alpha})$, $\frac{1}{h_2} \in \mult(\har{\beta})$, and $h = \tilde{h} h_1 h_2$, then $T_{\alpha,\beta,h}$ and $T_{\alpha,\beta,\tilde{h}}$ are similar.
\end{rem}

\begin{proof}
  To see this, take the invertible operator
  $$ X=\begin{pmatrix}
    (M_{\alpha,h_1}^{-1})^* & 0 \\
    0 & M_{\beta,h_2}^* \\ \end{pmatrix}, $$
  then, by simple calculation, $ X T_{\alpha,\beta,h} X^{-1} = T_{\alpha,\beta,\tilde{h}} $,
  and this implies that $T_{\alpha,\beta,h}$ and $T_{\alpha,\beta,\tilde{h}}$ are similar.
\end{proof}

\section{The case of $B(z) = e^{i \theta} \frac{z-a}{1-\bar{a}z}$}\label{step 2}

  Recall a basic concept of weakly homogeneous operators, which was introduced by Clark and Misra \cite{Clark1993}. Denote by $\aut(\D)$ the analytic automorphism group of $\D$, which is the set of all analytic bijection from $\D$ to itself. As well know, a function $\phi$ is in $\aut(\D)$ if and only if it has the following form:
  $ \phi(z) = e^{i \theta} \frac{z-a}{1-\bar{a}z},\ z \in \D, $
  for some $\theta \in [0,2\pi)$ and some $a \in \D$. A bounded linear operator $T$ on a complex separable Hilbert space $H$ is called weakly homogeneous if $\sigma(T) \subseteq \DD$ and for any $\phi \in \aut(\D)$, $\phi(T)$ is similar to $T$.

  We say a weighted Hardy space $\har{\alpha}$ be \textbf{M$\ddot{o}$bius invariant} if for each $\phi \in \aut(\D)$, $f\circ\phi \in \har{\alpha}$ whenever $f \in \har{\alpha}$. Let $\har{\alpha}$ be a M$\ddot{o}$bius invariant weighted Hardy space. Then for each $\phi \in \aut(\D)$, one can define a composition operator $C_{\alpha,\phi} \colon f \in \har{\alpha} \to f \circ \phi \in \har{\alpha}$, which is bounded by closed graph Theorem. What's more, $C_{\alpha,\phi}$ is invertible and $C_{\alpha,\phi}^{-1}=C_{\alpha,\phi^{-1}}$. By simple calculation, $\phi(M_{\alpha,z}) C_{\alpha,\phi} = M_{\alpha,\phi} C_{\alpha,\phi} = C_{\alpha,\phi} M_{\alpha,z}$. Thus, $M_{\alpha,z}$ is a weakly homogeneous operator.

\begin{rem}\label{example 1 rem 2}
  It is known that for each $\lambda \in \R$, the Hilbert space $\mathrm{H}_{(\lambda)}$ is M$\ddot{o}$bius invariant (see \cite{Zorboska1990}, \cite{COWEN1995}). 
\end{rem}

\begin{rem}\label{example 2 rem 1}
  Theorem 3.1 in \cite{hou2023} shows that each weighted Hardy space of polynomial growth is M$\ddot{o}$bius invariant. 
\end{rem}

  The main idea of the next proposition originates from Theorem 3.16 in \cite{Ghara2018}. Here we use Proposition \ref{prop 3.4} to improve it in part.

\begin{prop}\label{prop 3.5}
  Let $\har{\alpha}$ and $\har{\beta}$ be two M$\ddot{o}$bius invariant weighted Hardy spaces. Suppose the weight sequences $\alpha$ and $\beta$ satisfy condition (\ref{cond 1}). If the operator $T_{\alpha,\beta,h} $ is weakly homogeneous, where $h \in \mathcal{C}(\DD) \cap \mult(\har{\alpha},\har{\beta})$, then $h$ is zero everywhere on $\DD$ or $h$ is non-zero everywhere on $\DD$.
\end{prop}

\begin{proof}
  If $h$ is a zero function on $\D$, then $h$ is a zero function on $\DD$ since $h \in \mathcal{C}(\DD)$. So we suppose that $h$ is a non-zero function on $\D$. 

  Let $\phi \in \aut(\D)$ and let $\tilde{\phi}$ be the function defined by $\tilde{\phi}(z)=\overline{\phi(\overline{z})}$. Then $\tilde{\phi}$ is also in $\aut(\D)$. Since $T := T_{\alpha,\beta,h}$ is weakly homogeneous, $\tilde{\phi}(T) \sim T$. From $M_{\alpha,z}^* M_{\alpha,\beta,h}^* = M_{\alpha,\beta,h}^* M_{\beta,z}^*$, it can be calculated that
  \begin{equation*}
      \tilde{\phi}(T)
      = \begin{pmatrix}
        \tilde{\phi}(M_{\alpha,z}^*) & \tilde{\phi}'(M_{\alpha,z}^*) M_{\alpha,\beta,h}^* \\
        0 & \tilde{\phi}(M_{\beta,z}^*) \\ \end{pmatrix}\\
      = \begin{pmatrix}
        M_{\alpha,\phi}^* & M_{\alpha,\beta,h \phi'}^* \\
        0 & M_{\beta,\phi}^* \\ \end{pmatrix}.
  \end{equation*}
  M$\ddot{o}$bius invariance guarantees that the composition operators $C_{\alpha,\phi}$ and $C_{\beta,\phi}$ are invertible and that $C_{\alpha,\phi}^{-1}=C_{\alpha,\phi^{-1}}$ and $C_{\beta,\phi}^{-1}=C_{\beta,\phi^{-1}}$. Then, by simple calculation,
  \begin{equation*}
    \begin{pmatrix}
      C_{\alpha,\phi}^* & 0 \\
      0 & C_{\beta,\phi}^* \\
    \end{pmatrix}
    \tilde{\phi}(T) \begin{pmatrix}
      C_{\alpha,\phi}^* & 0 \\
      0 & C_{\beta,\phi}^* \\ \end{pmatrix}^{-1}\\
    = \begin{pmatrix}
      M_{\alpha,z}^* & M_{\alpha,\beta,(h\circ\phi^{-1})(\phi'\circ\phi^{-1})}^* \\
      0 & M_{\beta,z}^* \\ \end{pmatrix}.
  \end{equation*}
  So
  $$ T = T_{\alpha,\beta,h} \sim \tilde{\phi}(T) \sim T_{\alpha,\beta,(h\circ\phi^{-1})(\phi'\circ\phi^{-1})}. $$

  Since $\phi$ is a bijection from $\D$ to itself and $\phi'(z) = e^{i \theta} \frac{1-\abs{a}^2}{(1-\bar{a}z)^2}$ is non-zero everywhere on $\D$, then $(h\circ\phi^{-1})(\phi'\circ\phi^{-1})$ is a non-zero function on $\D$. Consequently, according to Proposition \ref{prop 3.4}, there exists $h_1 = h_{1,\phi} \in \mult(\har{\alpha})$ and $h_2 = h_{2,\phi} \in \mult(\har{\beta})$ such that $\frac{1}{h_1} \in \mult(\har{\alpha})$, $\frac{1}{h_2} \in \mult(\har{\beta})$, and $ h = (h\circ\phi^{-1})(\phi'\circ\phi^{-1}) h_1 h_2 $. Thus,
  \begin{equation}\label{3.5.1}
    h(\phi(z)) = h(z) \phi'(z) h_1(\phi(z)) h_2(\phi(z)),\ z \in \D.
  \end{equation}

  Then, similar to the last part of the proof of Theorem 3.16 in \cite{Ghara2018}, we can obtain that $h$ is non-zero everywhere on $\DD$.
\end{proof}
%

\begin{rem}\label{prop 3.5 rem}
  Let $\har{\alpha}$ and $\har{\beta}$ be two M$\ddot{o}$bius invariant weighted Hardy spaces. Let $h$ be in $\hol(\DD) \cap \mult(\har{\alpha},\har{\beta})$. If $h$ is zero everywhere on $\DD$ or $h$ is non-zero everywhere on $\DD$, then $T = T_{\alpha,\beta,h}$ is weakly homogeneous.
\end{rem}

\begin{proof}
  It is trivial if $h$ is zero everywhere on $\DD$. So we suppose that $h$ is non-zero everywhere on $\DD$. Then $\frac{1}{h} \in \hol(\DD)$. Let $\phi \in \aut(\D)$ and let $\tilde{\phi}$ be the function defined by $\tilde{\phi}(z)=\overline{\phi(\overline{z})}$.
  From the proof of Proposition \ref{prop 3.5}, we see that $ \tilde{\phi}(T) \sim T_{\alpha,\beta,(h\circ\phi^{-1})(\phi'\circ\phi^{-1})} $. It is not hard to see that the function $\tilde{h} := (h\circ\phi^{-1})(\phi'\circ\phi^{-1})$ is in $\hol(\DD)$ and non-zero everywhere on $\DD$. Then $\frac{1}{\tilde{h}} \in \hol(\DD)$. Since $h = \tilde{h} \frac{1}{\tilde{h}} h$ and $\hol(\DD)$ is contained in both $\mult(\har{\alpha})$ and $\mult(\har{\beta})$, then according to Remark \ref{prop 3.4 rem}, $ T = T_{\alpha,\beta,h} \sim T_{\alpha,\beta,(h\circ\phi^{-1})(\phi'\circ\phi^{-1})} $. Thus, $\tilde{\phi}(T) \sim T$. Since $\tilde{\phi}$ is arbitrary in $\aut(\D)$, we obtain that $T$ is weakly homogeneous.
\end{proof}

\section{The case of $B(z) = z^n$}\label{step 3}

\subsection{Power of a special operator}

  Let $\har{\alpha}$ be a weighted Hardy space and $n$ be a positive integer. Denote
  $$ \har{\alpha,j} := \vee \set{z^{j+kn}}_{k=0}^{\infty} \subseteq \har{\alpha},\ j=0,1,\dots,n-1. $$
  Then, it can be verified that
  $$ \har{\alpha} = \har{\alpha,0} \dotplus \har{\alpha,1} \dotplus \cdots \dotplus \har{\alpha,n-1} $$
  is a Hilbert direct sum, and $\set{z^{j+kn}}_{k=0}^{\infty}$ is an orthogonal basis of $\har{\alpha,j}$. What's more, $\har{\alpha,j}$ is an invariant subspace of the multiplication operator $M_{\alpha,z^n} = M_{\alpha,z}^n$.

  We say that a weighted Hardy space $\har{\alpha}$ has \textbf{property A} for a positive integer $n$, if there exists $c_{1}=c_{1}(n)>0$ and $c_{2}=c_{2}(n)>0$ such that
  $$ \alpha(k) \leq c_1 \alpha(j+kn) $$
  and
  $$ \alpha(j+kn) \leq c_2 \alpha(k) $$
  for any $j=0,1,\dots,n-1$ and $k=0,1,\dots$.

  Inspired by \cite{Ahmadi2012}, we give the following lemma.

\begin{lem}\label{lem 3.2}
  Let $\har{\alpha}$ be a weighted Hardy space with property A for a positive integer $n$. Then for each $j=0,1,\dots,n-1$, the map
  $$ X_{\alpha,j} \colon f = \sum_{k=0}^{\infty} \hat{f}(k) z^k \in \har{\alpha} \to \sum_{k=0}^{\infty} \hat{f}(k) z^{j+kn} \in \har{\alpha,j} $$
  is an invertible bounded linear operator. Moreover, for each $j=0,1,\dots,n-1$, $M_{\alpha,z^n}|_{\har{\alpha,j}}$ $X_{\alpha,j}$ $=$ $X_{\alpha,j}$ $M_{\alpha,z}$. Thus, $M_{\alpha,z^n}$ is similar to $\oplus_{1}^{n} M_{\alpha,z}$.
\end{lem}

\begin{proof}
  Let $f = \sum_{k=0}^{\infty} \hat{f}(k) z^k \in \har{\alpha}$, then
  \begin{equation*}
    \begin{split}
      \norm{\sum_{k=0}^{\infty} \hat{f}(k) z^{j+kn}}^2 &= \sum_{k=0}^{\infty} \abs{\hat{f}(k)}^2 \alpha(j+kn)^2 \\
        &\leq c_2^2 \sum_{k=0}^{\infty} \abs{\hat{f}(k)}^2 \alpha(k)^2
        = c_2^2 \norm{f}^2, 
    \end{split}
  \end{equation*}
  which implies $g = \sum_{k=0}^{\infty} \hat{f}(k) z^{j+kn} \in \har{\alpha,j}$ and $\norm{g} \leq c_2 \norm{f}$. Thus, it can be seen that $X_{\alpha,j}$ is not only well defined but also a bounded linear operator. Obviously, $X_{\alpha,j}$ is an injection.

  Let $g = \sum_{k=0}^{\infty} \tilde{g}(k) z^{j+kn} \in \har{\alpha,j}$, then
  \begin{equation*}
    \begin{split}
      \norm{\sum_{k=0}^{\infty} \tilde{g}(k) z^k}^2 &= \sum_{k=0}^{\infty} \abs{\tilde{g}(k)}^2 \alpha(k)^2 \\
        &\leq c_1^2 \sum_{k=0}^{\infty} \abs{\tilde{g}(k)}^2 \alpha(j+kn)^2
        = c_1^2 \norm{g}^2, 
    \end{split}
  \end{equation*}
  which implies that $f = \sum_{k=0}^{\infty} \tilde{g}(k) z^k \in \har{\alpha}$ and $g = X_{\alpha,j} f$. Thus, $X_{\alpha,j}$ is a surjection. 
  Therefore, $X_{\alpha,j}$ an invertible operator.

  The remaining conclusions can be easily checked.
\end{proof}

\begin{rem}\label{property A rem}
  Form the assumption (\ref{equation 1.1}), we can see that a weighted Hardy space $\har{\alpha}$ has property A for a positive integer $n$ if and only if there exists $c_{1}=c_{1}(n)>0$ and $c_{2}=c_{2}(n)>0$ such that
  $$ c_{1} \leq \frac{\alpha(n-1+kn)}{\alpha(k)} \leq c_{2}  $$
  for any $k=0,1,\dots$.
\end{rem}

\begin{rem}\label{example 1 rem 3}
  For each $\lambda \in \R$, the weighted Hardy spaces $\har{\alpha_{\lambda}} = \mathrm{H}_{(\lambda)}$ has property A for any positive integer $n$. In fact, this follows Remark \ref{property A rem} since
  $$ \frac{\alpha_{\lambda}(n-1+kn)}{\alpha_{\lambda}(k)} = n^\lambda $$
  for any $k=0,1,\dots$.
\end{rem}

\begin{rem}\label{example 2 rem 2}
  Each weighted Hardy space $\har{\alpha}$ of polynomial growth has property A for any positive integer $n$.
\end{rem}

\begin{proof}
  The weighted Hardy space $\har{\alpha}$ is of polynomial growth means that there exists $N \in \N$ such that for each $k \in \N$,
  $$ \frac{k+1}{k+N+1} \leq \frac{\alpha(k)}{\alpha(k-1)} \leq \frac{k+N+1}{k+1} $$
  (see \cite{hou2023}).
  Then for each $k=0,1,\dots$, whenever $n-1+kn>k$,
  $$ \prod_{j=k+1}^{n-1+kn} \frac{j+1}{j+N+1} \leq \frac{\alpha(n-1+kn)}{\alpha(k)} = \prod_{j=k+1}^{n-1+kn} \frac{\alpha(j)}{\alpha(j-1)} \leq \prod_{j=k+1}^{n-1+kn} \frac{j+N+1}{j+1}. $$
  Form
  \begin{equation*}
    \prod_{j=k+1}^{n-1+kn} \frac{j+N+1}{j+1}
    = \prod_{j=1}^{N} \frac{n(k+1)+j}{k+1+j}
    \leq \prod_{j=1}^{N} n
    = n^N,
  \end{equation*}
  we get
  \begin{equation*}
    \frac{1}{n^N} \leq \frac{\alpha(n-1+kn)}{\alpha(k)} \leq n^N.
  \end{equation*}
  The last inequality holds for each $k=0,1,\dots$. Thus, this remark follows Remark \ref{property A rem}.
\end{proof}

  Let $\har{\alpha}$ and $\har{\beta}$ be two weighted Hardy spaces. It is not hard to check that the following conditions are equivalent:

  (a) $\har{\alpha} \subseteq \har{\beta}$;

  (b) $1 \in \mult(\har{\alpha},\har{\beta})$;

  (c) $\sup_{k \geq 0} \frac{\beta(k)}{\alpha(k)} < \infty$.

  If one of the conditions above is true, the multiplication operator $M_{\alpha,\beta,1}$ is the inclusion mapping $\iota \colon \har{\alpha} \to \har{\beta}$. In addition, $\hol(\DD) \subseteq \mult(\har{\alpha}) \subseteq \mult(\har{\alpha},\har{\beta})$.

\begin{rem}\label{example 1 rem 4}
  Let $\lambda,\mu \in \R$. The weighted Hardy spaces $\har{\alpha_{\lambda}} = \mathrm{H}_{(\lambda)}$ and $\har{\alpha_{\mu}} = \mathrm{H}_{(\mu)}$ satisfy $\mathrm{H}_{(\lambda)} \subseteq \mathrm{H}_{(\mu)}$ if and only if $\lambda - \mu \geq 0$.
\end{rem}

\begin{prop}\label{prop 3.6}
  Let $\har{\alpha}$ and $\har{\beta}$ be two weighted Hardy spaces with property A for a positive integer $n$. Suppose $\har{\alpha} \subseteq \har{\beta}$. Denote $T = T_{\alpha,\beta,1}$, then $T^n$ is similar to $T \oplus (\oplus_{j=1}^{n-1} \tilde{T})$, where $\tilde{T} = T_{\alpha,\beta,z}$.
\end{prop}

\begin{proof}
  It is enough to prove $(T^*)^n \sim T^* \oplus (\oplus_{j=1}^{n-1} \tilde{T}^*)$, where
  $$ T^*=\begin{pmatrix}
           M_{\alpha,z} & 0 \\
           M_{\alpha,\beta,1} & M_{\beta,z} \\ \end{pmatrix},\
  \tilde{T}^*=\begin{pmatrix}
           M_{\alpha,z} & 0 \\
           M_{\alpha,\beta,z} & M_{\beta,z} \\ \end{pmatrix}. $$

  From $M_{\alpha,\beta,1} M_{\alpha,z} = M_{\beta,z} M_{\alpha,\beta,1}$, it can be calculated that
  $$ (T^*)^n=\begin{pmatrix}
               M_{\alpha,z^n} & 0 \\
               n M_{\alpha,\beta,z^{n-1}} & M_{\beta,z^n} \end{pmatrix}. $$
  Since
  $$ \begin{pmatrix}
       I & 0 \\
       0 & \frac{1}{n} I \\ \end{pmatrix}
     (T^*)^n \begin{pmatrix}
       I & 0 \\
       0 & \frac{1}{n} I \\ \end{pmatrix}^{-1}
    = \begin{pmatrix}
       M_{\alpha,z^n} & 0 \\
       M_{\alpha,\beta,z^{n-1}} & M_{\beta,z^n} \\ \end{pmatrix}
    =: A, $$
  it followed that $(T^*)^n \sim A$.

  Let
  \begin{align*}
    K_0  &= \har{\alpha,0} \oplus \har{\beta,n-1}, \\
    K_1 &= \har{\alpha,1} \oplus \har{\beta,0}, \\
    & \cdots \\
    K_{n-1} &= \har{\alpha,n-1} \oplus \har{\beta,n-2},
  \end{align*}
  then
  $$ \har{\alpha} \oplus \har{\beta} = K_0 \dotplus K_1 \dotplus \cdots \dotplus K_{n-1} $$
  is Hilbert direct sum. One can check that $\har{\alpha,j}$ and $\har{\beta,j}$ are invariant subspace of $M_{\alpha,z^n}$ and $M_{\beta,z^n}$, respectively, and that
  \begin{align*}
    M_{\alpha,\beta,z^{n-1}}(\har{\alpha,0}) & \subseteq \har{\beta,n-1}, \\
    M_{\alpha,\beta,z^{n-1}}(\har{\alpha,1}) & \subseteq \har{\beta,0}, \\
    \cdots \\
    M_{\alpha,\beta,z^{n-1}}(\har{\alpha,n-1}) & \subseteq \har{\beta,n-2}.
  \end{align*}
  Hence, $K_0,K_1,\dots,K_{n-1}$ are invariant subspace of $A$.

  To complete the proof, it is enough to prove that
  $$ A|_{K_0} \sim T^*,\; A|_{K_1} \sim \tilde{T}^*,\; \dots,\; A|_{K_{n-1}} \sim \tilde{T}^*. $$ 
  This implies that $(T^*)^n \sim A \sim T^* \oplus (\oplus_{j=1}^{n-1} \tilde{T}^*)$.

It is not hard to see that
  \begin{align*}
    A|_{K_0} &= \begin{pmatrix}
                 M_{\alpha,z^n}|_{\har{\alpha,0}} & 0 \\
                 P_{\har{\beta,n-1}} M_{\alpha,\beta,z^{n-1}} P_{\har{\alpha,0}} & M_{\beta,z^n}|_{\har{\beta,n-1}} \\ \end{pmatrix}, \\
    A|_{K_1} &= \begin{pmatrix}
                 M_{\alpha,z^n}|_{\har{\alpha,1}} & 0 \\
                 P_{\har{\beta,0}} M_{\alpha,\beta,z^{n-1}} P_{\har{\alpha,1}} & M_{\beta,z^n}|_{\har{\beta,0}} \\ \end{pmatrix}, \\
             & \cdots \\
    A|_{K_{n-1}} &= \begin{pmatrix}
                 M_{\alpha,z^n}|_{\har{\alpha,n-1}} & 0 \\
                 P_{\har{\beta,n-2}} M_{\alpha,\beta,z^{n-1}} P_{\har{\alpha,n-1}} & M_{\beta,z^n}|_{\har{\beta,n-2}} \\ \end{pmatrix}.
  \end{align*}
  Lemma \ref{lem 3.2} shows that for any $j=0,1,\dots,n-1$, the maps $ X_{\alpha,j} \colon \har{\alpha}$$ \to \har{\alpha,j} $ and $ X_{\beta,j} \colon \har{\beta}$$ \to \har{\beta,j} $
  are invertible operators. Take the invertible operators
  $$ X_0 = \begin{pmatrix}
             X_{\alpha,0} & 0 \\
             0 & X_{\beta,n-1} \\ \end{pmatrix},\,
     X_1 = \begin{pmatrix}
             X_{\alpha,1} & 0 \\
             0 & X_{\beta,0} \\ \end{pmatrix},\,
     \dots,\,
     X_{n-1} = \begin{pmatrix}
             X_{\alpha,n-1} & 0 \\
             0 & X_{\beta,n-2} \\ \end{pmatrix}. $$ 
  By calculation we obtain that $M_{\alpha,z^n}|_{\har{\alpha,j}} X_{\alpha,j} = X_{\alpha,j} M_{\alpha,z}$, $M_{\beta,z^n}|_{\har{\beta,j}} X_{\beta,j} = X_{\beta,j} M_{\beta,z}$ and that
  \begin{align*}
    P_{\har{\beta,n-1}} M_{\alpha,\beta,z^{n-1}} P_{\har{\alpha,0}} X_{\alpha,0} &= X_{\beta,n-1} M_{\alpha,\beta,1}, \\
    P_{\har{\beta,0}} M_{\alpha,\beta,z^{n-1}} P_{\har{\alpha,1}} X_{\alpha,1} &= X_{\beta,0} M_{\alpha,\beta,z}, \\
    & \cdots \\
    P_{\har{\beta,n-2}} M_{\alpha,\beta,z^{n-1}} P_{\har{\alpha,n-1}} X_{\alpha,n-1} &= X_{\beta,n-2} M_{\alpha,\beta,z}. \\
  \end{align*}
  The above equation indicates that
  $$ A|_{K_0} X_0 = X_0 T^*,\ A|_{K_1} X_1 = X_1 \tilde{T}^*,\ \dots,\ A|_{K_{n-1}} X_{n-1} = X_{n-1} \tilde{T}^*, $$ 
  and hence
  $$ A|_{K_0} \sim T^*,\; A|_{K_1} \sim \tilde{T}^*,\; \dots,\; A|_{K_{n-1}} \sim \tilde{T}^*. $$ 
\end{proof}

\subsection{Dissimilarity}

  In the remaining part of this section, we need some conclusions related to strongly irreducible operators (`strongly irreducible' is also abbreviated as `(SI)').


  Denote $\mathrm{K}_0 (\mathcal{B})$ for the $\mathrm{K}_0$-group of a Banach algebra $\mathcal{B}$.
  For a bounded linear operator $A$ on a Hilbert space $H$, denote $\set{A}'$ for the commutant algebra of $A$ and $A^{(n)}$ for the $n$ copies $\oplus_n^{n} A$ of $A$.
  Then the following theorem is Theorem 1 in \cite{JIANG2005}.

\begin{thm}\label{Jiang}
  Let $T = A_1^{(n_1)} \oplus A_2^{(n_2)} \oplus \cdots \oplus A_k^{(n_k)}$, where $A_1,A_2,\dots,A_k$ are strongly irreducible Cowen-Douglas operators, $A_i$ and $A_j$ are not similar whenever $i \neq j$, and $n_1,n_2,\dots,n_k$ are positive integers. Then
  $ \mathrm{K}_0 (\set{T}') \cong \mathbb{Z}^k $.
\end{thm}



  The following proposition, which is Proposition 2.22 in \cite{JI2017}, gives a characterization of strong irreducibility in $FB_2(\Omega)$.

\begin{lem}\label{lem 3.3}
  An operator
  $$ T=\begin{pmatrix}
    T_0 & S \\
    0 & T_1 \end{pmatrix} $$
  in $FB_2(\Omega)$ is strongly irreducible if and only if $S \notin \Ran\sigma_{T_0,T_1}$.
\end{lem}

\begin{lem}\label{lem 3.4}
  Let $\har{\alpha}$ and $\har{\beta}$ be two weighted Hardy spaces, where the weight sequences $\alpha$ and $\beta$ satisfy condition (\ref{cond 1}). Then

  (a) for any non-zero function $h$ in $\mult(\har{\alpha},\har{\beta})$, the operator $T_{\alpha,\beta,h}$ is strongly irreducible.

  (b) for any non-zero functions $h$ and $\tilde{h}$ in $\mult(\har{\alpha},\har{\beta})$ such that $h$ and $\tilde{h}$ have different zero points on $\D$, the operators $T_{\alpha,\beta,h}$ and $T_{\alpha,\beta,\tilde{h}}$ are not similar.
\end{lem}

\begin{proof}
  (a) Since $\sigma_{M_{\alpha,z}^*,M_{\beta,z}^*} (M_{\alpha,\beta,h}^*) =0$, it follows that $M_{\alpha,\beta,h}^* \notin \Ran\sigma_{M_{\alpha,z}^*,M_{\beta,z}^*}$ according to
  Proposition \ref{prop 3.2}. Then $T_{\alpha,\beta,h}$ is strongly irreducible by Lemma \ref{lem 3.3}.

  (b) Clearly, this is a corollary of Proposition \ref{prop 3.4}.
\end{proof}

\vskip 5pt

\begin{prop}\label{prop 3.9}
  Let $\har{\alpha}$ and $\har{\beta}$ be two weighted Hardy spaces with property A for a positive integer $n \geq 2$. Suppose $\har{\alpha} \subseteq \har{\beta}$ and suppose the weight sequences $\alpha$ and $\beta$ satisfy condition (\ref{cond 1}). Denote $T = T_{\alpha,\beta,1}$, then $T^n$ is not similar to $\oplus_1^n T$.
\end{prop}

\begin{proof}
  Proposition \ref{prop 3.6} shows that $T^n \sim T \oplus (\oplus_{j=1}^{n-1} \tilde{T}) = T^{(1)} \oplus \tilde{T}^{(n-1)}$, where $\tilde{T} = T_{\alpha,\beta,z}$. Lemma \ref{lem 3.4} shows that $T$ and $\tilde{T}$ are strongly irreducible Cowen-Douglas operators and $T \nsim \tilde{T}$.
  Assume $T^n$ is similar to $\oplus_1^n T = T^{(n)}$, then
  $$ \mathrm{K}_0 (\set{T^{(1)} \oplus \tilde{T}^{(n-1)}}') \cong \mathrm{K}_0 (\set{T^n}') \cong \mathrm{K}_0 (\set{T^{(n)}}'). $$
  But from Theorem \ref{Jiang}, it follows that
  $$ \mathrm{K}_0 (\set{T^{(1)} \oplus \tilde{T}^{(n-1)}}') \cong \mathbb{Z}^2,\ \ \mathrm{K}_0 (\set{T^{(n)}}') \cong \mathbb{Z}. $$
  So we get a contradiction since $\mathbb{Z}^2 \ncong \mathbb{Z}$.
\end{proof}

\section{Main results}

  In this section, we complete the proof of the main theorem and give some concrete examples.

\begin{proof}[Proof of Theorem \ref{thm 4.1}]
  Suppose $h \neq 0$. Since the functions in $\aut(\D)$ are finite Blaschke products with order $1$, $T$ is weakly homogeneous. Since $\sup_{k \geq 0} \frac{\beta(k)}{\alpha(k)} < \infty$, i.e. $\har{\alpha} \subseteq \har{\beta}$, $\hol(\DD) \subseteq \mult(\har{\alpha}) \subseteq \mult(\har{\alpha},\har{\beta})$. Hence, according to Proposition \ref{prop 3.5}, $h$ is non-zero everywhere on $\DD$, which implies $\frac{1}{h} \in \hol(\DD)$. Thus, $h,\frac{1}{h} \in \mult(\har{\alpha})$. From Remark \ref{prop 3.4 rem}, we have $T = T_{\alpha,\beta,h} \sim T_{\alpha,\beta,1} =: \hat{T}$. Since $B(z) = z^{n_0}, z \in \D$ is a finite Blaschke product with order $n_0$, it follows that $T^{n_0} = B(T) \sim \oplus_1^{n_0} T$ and hence $\hat{T}^{n_0} \sim \oplus_1^{n_0} \hat{T}$. But this contradicts the conclusion of Proposition \ref{prop 3.9} that $\hat{T}^{n_0} \nsim \oplus_1^{n_0} \hat{T}$.
\end{proof}

\begin{cor}
  Let $\lambda - \mu \in [0,1)$. Let
  $$ T = \begin{pmatrix}
    M_{z}^* & M_{h}^* \\
    0 & M_{z}^* \\ \end{pmatrix}, $$
  on $\mathrm{H}_{(\lambda)} \oplus \mathrm{H}_{(\mu)}$, where $h \in \hol(\DD)$. If for any finite Blaschke product $B$, $B(T)$ is similar to $\oplus_1^n T$, where $n$ is the order of $B$, then $h=0$.
\end{cor}

\begin{proof}
  This corollary now follows Remarks \ref{example 1 rem 1}, \ref{example 1 rem 2}, \ref{example 1 rem 3}, \ref{example 1 rem 4} and Theorem \ref{thm 4.1}.
\end{proof}

\begin{cor}
  Let $\har{\alpha}$ and $\har{\beta}$ be two weighted Hardy spaces of polynomial growth. Suppose the weight sequences $\alpha$ and $\beta$ satisfy
  $$ \sup_{k \geq 0} \frac{\beta(k)}{\alpha(k)} < \infty \quad \text{and} \quad \lim_{k \to \infty} \frac{\alpha(k)}{k\beta(k)} = 0. $$
  Denote $T = T_{\alpha,\beta,h}$, where $h \in \hol(\DD)$. If for any finite Blaschke product $B$, $B(T)$ is similar to $\oplus_1^n T$, where $n$ is the order of $B$, then $h=0$.
\end{cor}

\begin{proof}
  This corollary now follows Remark \ref{example 2 rem 1}, \ref{example 2 rem 2} and Theorem \ref{thm 4.1}.
\end{proof}

  For many weighted Hardy space $\har{\alpha}$, the multiplier algebra $\mult(\har{\alpha})$ has better properties, such as $\mult(\har{\alpha}) = \mathrm{H}^{\infty}(\D)$. In Theorem \ref{thm 4.1}, if a higher condition is required for multiplier algebra of weighted Hardy space, a stronger conclusion can be obtained.

\begin{thm}\label{thm 4.2}
  Let $\har{\alpha}$ and $\har{\beta}$ be two M$\ddot{o}$bius invariant weighted Hardy spaces with property A for some integer $n_0 \geq 2$. Assume $\mathcal{C}(\DD) \cap \hol(\D) \subseteq \mult(\har{\alpha})$. Suppose the weight sequences $\alpha$ and $\beta$ satisfy
  $$ \sup_{k \geq 0} \frac{\beta(k)}{\alpha(k)} < \infty \quad \text{and} \quad \lim_{k \to \infty} \frac{\alpha(k)}{k\beta(k)} = 0. $$
  Denote $T = T_{\alpha,\beta,h}$, where $h \in \mathcal{C}(\DD) \cap \hol(\D)$. If for any finite Blaschke product $B$, $B(T)$ is similar to $\oplus_1^n T$, where $n$ is the order of $B$, then $h=0$.
\end{thm}

\begin{proof}
  Similar to the proof of Theorem \ref{thm 4.1}.
\end{proof}

\begin{cor}
  Let $\mathrm{H}^2(\D)$ be the classical Hardy space and $\mathrm{L}^2_a(\D)$ be the classical Bergman space. Let
  $$ T = \begin{pmatrix}
    M_{z}^* & M_{h}^* \\
    0 & M_{z}^* \\ \end{pmatrix}, $$
  on $\mathrm{H}^2(\D) \oplus \mathrm{L}^2_a(\D)$, where $h \in \mathcal{C}(\DD) \cap \hol(\D)$. If for any finite Blaschke product $B$, $B(T)$ is similar to $\oplus_1^n T$, where $n$ is the order of $B$, then $h=0$.
\end{cor}

\begin{proof}
  As will known, $\mult(\mathrm{H}^2(\D))$ and $\mult(\mathrm{L}^2_a(\D))$ are both $H^{\infty}(\D)$. Note that $\mathrm{H}^2(\D)$ is $\mathrm{H}_{(0)}$ and $\mathrm{L}^2_a(\D)$ is $\mathrm{H}_{(-\frac{1}{2})}$. Then this corollary now follows Remarks \ref{example 1 rem 1}, \ref{example 1 rem 2}, \ref{example 1 rem 3}, \ref{example 1 rem 4} and Theorem \ref{thm 4.2}.
\end{proof}

\bibliographystyle{amsplain}
\bibliography{YJ}

\end{document}